\documentclass[10pt]{article}
\usepackage{amsfonts}
\usepackage[dvips]{graphicx}
\usepackage{amssymb,color}
\usepackage[latin1]{inputenc}
\usepackage{epic}
\usepackage[latin1]{inputenc}
\usepackage{enumerate}
\usepackage{color}
\usepackage{lettrine}
\usepackage{calc,pifont}
\usepackage[noindentafter,calcwidth]{titlesec}
\usepackage{amsmath}
\usepackage{amssymb}
\usepackage{amsthm}
\usepackage{subfigure}
\usepackage{lineno} 

\def\BBox{\kern  -0.2cm\hbox{\vrule width 0.2cm height 0.2cm}}

\newtheorem{lemma}{Lemma}[section]

\newtheorem{theorem}{Theorem}[section]

\newtheorem{corollary}{Corollary}[section]

\title{A note on domination in intersecting linear systems}

\author{Adri{\' a}n V\'azquez-\'Avila  \footnotemark[1]}
\date{}

\begin{document}
\maketitle

\def\thefootnote{\fnsymbol{footnote}}
\footnotetext[1]{Subdirecci{\' o}n de Ingenier{\' i}a y Posgrado, Universidad Aeronáutica en Quer\'etaro, Parque Aeroespacial de Quer\'etaro, 76278, Quer\'etaro, M\'exico, {\tt adrian.vazquez@unaq.edu.mx}.}

\begin{abstract}
A linear system is a pair $(P,\mathcal{L})$ where
$\mathcal{L}$ is a family of subsets on a ground finite set $P$ such that $|l\cap l^\prime|\leq 1$, for every $l,l^\prime \in \mathcal{L}$. The elements of $P$ and $\mathcal{L}$ are called points and lines, respectively, and the linear system is called intersecting if any pair of lines intersect in exactly one point.
A subset $D$ of points of a linear system $(P,\mathcal{L})$ is a dominating set of $(P,\mathcal{L})$ if for every $u\in P\setminus D$ there exists $v\in D$ such that $u,v\in l$, for some $l\in\mathcal{L}$. The cardinality of a minimum dominating set of a linear system $(P,\mathcal{L})$ is called domination number of $(P,\mathcal{L})$, denoted by $\gamma(P,\mathcal{L})$. On the other hand, a subset $R$ of lines of a linear system $(P,\mathcal{L})$ is a 2-packing if any three elements of $R$ do not have a common point (are triplewise disjoint). The cardinality of a maximum 2-packing of a linear system $(P,\mathcal{L})$ is called 2-packing number of $(P,\mathcal{L})$, denoted by $\nu_2(P,\mathcal{L})$.

It is known that for intersecting linear systems $(P,\mathcal{L})$ of rank $r$ it satisfies $\gamma(P,\mathcal{L})\leq r-1$. In this note, we prove if $q$ is an even prime power and $(P,\mathcal{L})$ is an intersecting linear system of rank $q+2$ satisfying $\gamma(P,\mathcal{L})=q+1$, then this linear system can be constructed from a spanning $(q+1)$-uniform intersecting linear subsystem $(P^\prime,\mathcal{L}^\prime)$ of the projective plane of order $q$ satisfying $\tau(P^\prime,\mathcal{L}^\prime)=\nu_2(P^\prime,\mathcal{L}^\prime)-1=q+1$.

\end{abstract}

{\bf Keywords.} Linear systems, domination,transversal, 2-packing, projective plane.

{\bf Math. Subj. Class.} 05C65, 05C69

\section{Introduction}\label{sec:intro}

A \emph{set system} is a pair $(X,\mathcal{F})$ where $%
\mathcal{F}$ is a finite family of subsets on a ground set $X$. A
set system can be also thought of as a hypergraph, where the elements of $X$ and $\mathcal{F}$ are called \emph{vertices} and \emph{hyperedges}, respectively. The set system $(X,\mathcal{F})$ is \emph{intersecting} if $E\cap F\neq\emptyset$, for for every pair of distinct subsets $E,F \in \mathcal{F}$. On the other hand, the set system $(X,\mathcal{F})$ is a \emph{linear system} if it satisfies $|E\cap F|\leq 1$, for every pair of distinct subsets $E,F \in \mathcal{F}$; and it is denoted by $(P,\mathcal{L})$. The elements of $P$ and $\mathcal{L}$ are called \emph{{\tiny }points} and \emph{lines} respectively. For the remainder of this work we will only consider linear systems, and most of the following definitions can be generalized for set systems.

The \emph{rank} of a linear system is the maximum size of a line. An \emph{$r$-uniform} li\-near system is a linear system such that all lines contains exactly $r$ points. In this context, a simple graph is a 2-uniform linear system. Throughout this paper, we will only consider linear systems of rank $r\geq2$.

Let $(P,\mathcal{L})$ be a linear system and $p\in P$ be a point. The \emph{degree} of $p$ is the number of lines containing $p$, demoted by $deg(p)$, the maximum degree overall points of the linear systems is denoted by $\Delta(P,\mathcal{L})$. A point of degrees $2$ and $3$ is called \emph{double point} and \emph{triple point} respectively. A point of degree zero is called an \emph{isolated point}. Two points $p,q\in P$ are \emph{adjacent} if there is a line $l\in\mathcal{L}$ such that $p,q\in l$.

A {\emph{linear subsystem}} $(P^{\prime },\mathcal{L}^{\prime })$ of
a linear system $(P,\mathcal{L})$ satisfies that for any line $l^\prime\in\mathcal{L}^\prime$ there exists a line $l\in\mathcal{L}$
such that $l^\prime=l\cap P^\prime$. The \emph{linear subsystem induced} by a set of lines $\mathcal{L}^{\prime}\subseteq \mathcal{L}$ is the linear subsystem
$(P^{\prime },\mathcal{L}^{\prime })$ where
$P^{\prime}=\bigcup_{l\in \mathcal{L}^{\prime }} l$. The linear subsystem   $(P^\prime,\mathcal{L}^\prime)$ of $(P,\mathcal{L})$ is called \emph{spanning linear subsystem} if $P^\prime=P$. Given a linear system $(P,\mathcal{L})$, and a point $p\in P$, the linear system obtained from $(P,\mathcal{L})$ by \emph{deleting point $p$} is the linear system $(P^{\prime },\mathcal{L}^{\prime })$ induced by $\mathcal{L}^{\prime }=\{l\setminus \{p\}: l\in \mathcal{L}\}$. Given a linear system $(P,\mathcal{L})$ and a line $l\in \mathcal{L}$, the linear system obtained from $(P,\mathcal{L})$ by \emph{deleting the line $l$} is the linear system $(P^{\prime },\mathcal{L}^{\prime })$ induced by $\mathcal{L}^{\prime }=
\mathcal{L}\setminus \{l\}$. Let $(P^{\prime},\mathcal{L}^{\prime })$ and $(P,\mathcal{L})$ be two linear systems. $(P^{\prime},\mathcal{L}^{\prime })$ and $(P,\mathcal{L})$ are isomorphic, denoted by $(P^{\prime },\mathcal{L}^{\prime })\simeq(P,\mathcal{L})$, if after deleting points of degree 1 or 0 from both, the systems $(P^{\prime},\mathcal{L}^{\prime })$ and $(P,\mathcal{L})$ are isomorphic as hypergraphs, see \cite{MR3727901}.

A subset $D$ of points of a linear system $(P,\mathcal{L})$ is a \emph{dominating set} of $(P,\mathcal{L})$ if for every $u\in P\setminus D$ there exists $v\in D$ such that $u$ and $v$ are adjacent. The minimum cardinality of a dominating set of a linear system $(P,\mathcal{L})$ is called \emph{domination number}, and it is denoted by $\gamma(P,\mathcal{L})$. Domination in set systems was introduced by Acharya \cite{Acharya} and studied further in \cite{Acharya2,Arumugam,Buj,Dong,Jose}. A subset  $T$ of points of a linear system $(P,\mathcal{L})$ is a \emph{transversal} of $(P,\mathcal{L})$ (also called \emph{vertex cover} or \emph{hitting set}) if $T\cap l\neq\emptyset$, for every line $l\in\mathcal{L}$. The minimum cardinality of a transversal of a linear system $(P,\mathcal{L})$ is called \emph{transversal number}, and it is denoted by $\tau(P,\mathcal{L})$. On the other hand, a subset $R$ of lines of a linear system $(P,\mathcal{L})$ is a \emph{$2$-packing} of $(P,\mathcal{L})$ if the elements of $R$ are triplewise disjoint, that is, if three elements are chosen in $R$ then they are not incidents in a common point. The \emph{2-packing number} of $(P,\mathcal{L})$ is the maximum cardinality of a 2-packing of $(P,\mathcal{L})$ and it is  denoted by $\nu_2(P,\mathcal{L})$. Transversals and 2-packings in linear systems was studied in \cite{CGCA,CCA,CA,MR3727901}, while domination and 2-packing in simple graphs was studied in \cite{Avila}.

In \cite{Kang} it was proved that, if $(P,\mathcal{L})$ is an intersecting linear system of rank $r\geq2$, then $\gamma(P,\mathcal{L})\leq r-1$. In \cite{Shan} a characterization of set systems $(X,\mathcal{F})$ holding the equality when $r=3$ was given. On the other hand, in \cite{Dong} it was shown that all intersecting linear systems $(P,\mathcal{L})$ of rank 4 satisfying $\gamma(P,\mathcal{L})=3$ can be constructed by the Fano plane. In this note, we prove if $q$ is an even prime power and $(P,\mathcal{L})$ is an intersecting linear system of rank $(q+2)$ satisfying $\gamma(P,\mathcal{L})=q+1$, then this linear system can be constructed from a spanning $(q+1)$-uniform intersecting linear subsystem $(P^\prime,\mathcal{L}^\prime)$ of the projective plane of order $q$ satisfying $\tau(P^\prime,\mathcal{L}^\prime)=\nu_2(P^\prime,\mathcal{L}^\prime)-1=q+1$. This result generalizes the main result given in \cite{Dong}.
\section{Previous results}

Let $\mathcal{I}_r$ be the family of linear systems $(P,\mathcal{L})$ of rank $r$ with $\gamma(P,\mathcal{L})=r-1$. To better understand the main result, we need the following:
\begin{lemma}\cite{Dong}\label{lemma:primero}
For every linear system $(P,\mathcal{L})\in\mathcal{I}_r$, there exists an $r$-uniform spanning linear subsystem $(P^*,\mathcal{L}^*)$ of $(P,\mathcal{L})$ such that every line in $\mathcal{L}^*$ contains one point of degree one. 	
\end{lemma}
It is denoted by $(P^*,\mathcal{L}^*)$ to be the $r$-uniform intersecting spanning linear subsystem of a linear system of $(P,\mathcal{L})\in\mathcal{I}_r$ obtained from the Lemma \ref{lemma:primero}. Further, let $(P^\prime,\mathcal{L}^\prime)$ be the $(r-1)$-uniform intersecting linear subsystem obtained from $(P^*,\mathcal{L}^*)$ by deleting the point of degree one of each line of $\mathcal{L}^*$, see \cite{Dong}.

\begin{lemma}\cite{Dong}\label{lema:igualdades}
For every linear system $(P,\mathcal{L})\in\mathcal{I}_r$ it satisfies $$\gamma(P,\mathcal{L})=\gamma(P^*,\mathcal{L}^*)=\tau(P^*,\mathcal{L}^*)=\tau(P^\prime,\mathcal{L}^\prime)=r-1.$$ 
\end{lemma}

\begin{lemma}\cite{Dong}\label{lemma:gradomaximo}
Let $(P,\mathcal{L})\in\mathcal{I}_r$ $(r\geq3)$ then every line of $(P^\prime,\mathcal{L}^\prime)$ has at most one point of degree 2 and $\Delta(P^\prime,\mathcal{L}^\prime)=r-1$.  	
\end{lemma}

\begin{lemma}\cite{Dong}\label{lemma:desigualdadesplanoproy}
Let $(P,\mathcal{L})\in\mathcal{I}_r$ $(r\geq3)$ then
\begin{center}
 $3(r-2)\leq|\mathcal{L}^\prime|\leq(r-1)^2-(r-1)+1$ and $|P^\prime|=(r-1)^2-(r-1)+1$,	
\end{center}
and so $\gamma(P^\prime,\mathcal{L}^\prime)=1$.
\end{lemma}

\begin{theorem}\cite{CGCA}\label{teo:segundopaper}
Let $(P,\mathcal{L})$ be a linear system and $p,q\in P$ be two points such that $deg(p)=\Delta(P,\mathcal{L})$ and $deg(q)=\max\{deg(x): x\in
P\setminus\{p\}\}$. If $|\mathcal{L}|\leq deg(p)+deg(q)+\nu_2(P,\mathcal{L})-3$, then $\tau(P,\mathcal{L})\leq\nu_2(P,\mathcal{L})-1$.
\end{theorem}
\section{Main result}

In this section we prove, if $q$ is an even prime power and $(P,\mathcal{L})$ is an intersecting linear system of rank $(q+2)$ satisfying $\gamma(P,\mathcal{L})=q+1$, then this linear system can be constructed from a spanning $(q+1)$-uniform intersecting linear subsystem $(P^\prime,\mathcal{L}^\prime)$ of the projective plane of order $q$ satisfying $\tau(P^\prime,\mathcal{L}^\prime)=\nu_2(P^\prime,\mathcal{L}^\prime)-1=q+1$. 

Recall that, a \emph{finite projective plane} (or merely \emph{projective plane}) is an uniform linear system satisfying that any pair of points have a common line, any pair of lines have a common point and there exist four points in general position (there are not three collinear points). It is well known that if $(P,\mathcal{L})$ is a projective plane then there exists a number $q\in\mathbb{N}$, called \emph{order of projective plane}, such that every point (line, respectively) of $(P,\mathcal{L})$ is incident to exactly $q+1$ lines (points, respectively), and $(P,\mathcal{L})$ contains exactly $q^2+q+1$ points (lines, respectively). In addition to this, it is well known that projective planes of order $q$, denoted by $\Pi_q$, exist when $q$ is a prime power. For more information about the existence and the unicity of projective planes see, for instance, \cite{B86,B95}. In \cite{MR3727901} it was proved, if $q$ is an even prime power then $\tau(\Pi_q)=\nu_2(\Pi_q)-1=q+1$, however, if $q$ is an odd prime power then $\tau(\Pi_q)=\nu_2(\Pi_q)=q+1$.

To prove the following lemma (Lemma \ref{lemma:igualdad}), we need a result shown in \cite{Avila2}. Therefore, we present the proof given in \cite{Avila2}.

	\begin{lemma}\cite{Avila2}\label{lemma:impar}
		Let $(P,\mathcal{L})$ be an $r$-uniform intersecting linear system with $r\geq2$ be an even integer. If $\nu_2(P,\mathcal{L})=r+1$ then $\tau(P,\mathcal{L})=\frac{r+2}{2}$. 
\end{lemma}
\begin{proof}[\cite{Avila2}]
	Let $(P,\mathcal{L})$ be a $r$-uniform intersecting linear system with $\nu_2(P,\mathcal{L})=r+1$, where $r\geq2$ is a even integer. Let $R=\{l_1,\ldots,l_{r+1}\}$ be a maximum 2-packing of $(P,\mathcal{L})$. Since $(P,\mathcal{L})$ is an intersecting linear system then $l_i\cap l_j\neq\emptyset$, for $1\leq i<j\leq r+1$, and hence $|l_i|=r$, for $i=1,\ldots,r+1$. Let $l\in\mathcal{L}\setminus R$. Since $(P,\mathcal{L})$ is an intersecting linear system it satisfies $l\cap l_i=l\cap l_i\cap l_{j_i}\neq\emptyset$, for $i=1,\ldots,r+1$ and for some $j_i\in\{1,\ldots,r+1\}\setminus\{i\}$, however, by the pigeonhole principle there are a line $l_s\in R$ such that $l\cap l_s=\emptyset$, since there are an odd number of lines in $R$ and by linearity of $(P,\mathcal{L})$. Therefore $\mathcal{L}=R$ and $\tau(P,\mathcal{L})=\lceil \nu_{2}/2\rceil=\frac{r+2}{2}$ (see \cite{CA}).
\end{proof}

\begin{lemma}\label{lemma:igualdad}
Let $r\geq2$ be an even integer . For every $(P,\mathcal{L})\in\mathcal{I}_{r+2}$ it satisfies $\nu_2(P^\prime,\mathcal{L}^\prime)= r+2$. Hence $\tau(P^\prime,\mathcal{L}^\prime)=\nu_2(P^\prime,\mathcal{L}^\prime)-1$.	
\end{lemma}
\begin{proof}Since $(P^\prime,\mathcal{L}^\prime)$ is an intersecting $(r+1)$-uniform linear system, then $|l^\prime|\geq\nu_2(P^\prime,\mathcal{L}^\prime)-1$, for any line $l^\prime\in\mathcal{L}^\prime$. Hence, $\nu_2(P^\prime,\mathcal{L}^\prime)\leq r+2$. If $r+2$ is odd integer, then by Lemma \ref{lemma:impar} it satisfies $\tau(P^\prime,\mathcal{L}^\prime)=\frac{r+3}{2}$, which is a contradiction, since $\tau(P^\prime,\mathcal{L}^\prime)=r+1$.
	
Let $p\in P^\prime$ such that $deg(p)=\Delta(P^\prime,\mathcal{L}^\prime)$, and let $\Delta^\prime(P^\prime,\mathcal{L}^\prime)=\max\{deg(x): x\in P^\prime\setminus\{p\}\}$. By Theorem \ref{teo:segundopaper} if $|\mathcal{L}^\prime|\leq\Delta(P^\prime,\mathcal{L}^\prime)+\Delta^\prime(P^\prime,\mathcal{L}^\prime)+\nu_2(P,\mathcal{L})-3\leq3r+1$ (see Lemma \ref{lemma:gradomaximo}) then $\tau(P^\prime,\mathcal{L}^\prime)\leq\nu_2(P^\prime,\mathcal{L}^\prime)-1$, which implies $\nu_2(P^\prime,\mathcal{L}^\prime)\geq r+2$, and the equality $\nu_2(P^\prime,\mathcal{L}^\prime)=r+2$ holds. Hence, by Lemma \ref{lemma:desigualdadesplanoproy}, if $|\mathcal{L}^\prime|=3r$ then $\nu_2(P^\prime,\mathcal{L}^\prime)=r+2$, which implies if $(P,\mathcal{L})\in\mathcal{I}_{r+2}$ then $\nu_2(P^\prime,\mathcal{L}^\prime)=r+2$, since if $(\hat{P},\mathcal{\hat{L}})$ is a spanning linear subsystem of $(P^\prime,\mathcal{L}^\prime)$ then $\nu_2(\hat{P},\mathcal{\hat{L}})\leq\nu_2(P^\prime,\mathcal{L}^\prime)$.
\end{proof}

\begin{theorem}\label{thm:main}
Let $q$ be an even prime power. For every $(P,\mathcal{L})\in\mathcal{I}_{q+2}$ the linear system $(P^\prime,\mathcal{L}^\prime)$ is a spanning $(q+1)$-uniform linear subsystem of $\Pi_q$ such that $\tau(P^\prime,\mathcal{L}^\prime)=\nu_2(P^\prime,\mathcal{L}^\prime)-1=q+1$ with $|\mathcal{L}^\prime|\geq3q$.	
\end{theorem}

\begin{proof}
	Let $(P,\mathcal{L})\in\mathcal{I}_{q+2}$. Then $(P^\prime,\mathcal{L}^\prime)$ is an $(q+1)$-uniform intersecting linear system with $|P^\prime|=q^2+q+1$ (by Lemma \ref{lemma:desigualdadesplanoproy}
	). Furthermore, if $|\mathcal{L}^\prime|=q^2+q+1$ then all points of $(P^\prime,\mathcal{L}^\prime)$ have degree $q+1$ (it is a consequence of Lemma \ref{lemma:desigualdadesplanoproy}, see \cite{Dong}). Since projective planes are dual systems, the 2-packing number coincides with the cardinality of an oval, which is the maximum number of points in general position (no three of them collinear), and it is equal to $q+2=\nu_2(P^\prime,\mathcal{L}^\prime)$ (Lemma \ref{lemma:igualdad}), when $q$ is even, see for example \cite{B95}. Hence, the linear system $(P^\prime,\mathcal{L}^\prime)$ is a projective plane of order $q$, $\Pi_q$. Therefore, if $(P,\mathcal{L})\in\mathcal{I}_{q+2}$ then $(P^\prime,\mathcal{L}^\prime)$ is a spanning linear subsystem of $\Pi_q$ satisfying $\tau(P^\prime,\mathcal{L}^\prime)=\nu_2(P^\prime,\mathcal{L}^\prime)-1=q+1$, with $|\mathcal{L}^\prime|\geq3q$ (see Lemma \ref{lemma:desigualdadesplanoproy}). 
\end{proof} 

The following is a straightforward consequence of Theorem \ref{thm:main} which is the main result given in \cite{Dong}.

\begin{corollary}
If $(P,\mathcal{L})\in\mathcal{I}_{4}$ then either $(P^\prime,\mathcal{L}^\prime)$ is the Fano plane, $\Pi_2$, or $(P^\prime,\mathcal{L}^\prime)$ is obtained from $\Pi_2$ by deleting any line. 
\end{corollary}

{\bf Acknowledgment}

The author would like to thank the referees for many constructive suggestions to improve this paper.

Research was partially supported by SNI and CONACyT.


\begin{thebibliography}{10}
\bibitem{Acharya}	
B.D. Acharya, \emph{Domination in hypergraphs }, AKCE J. Comb. {\bf 4} (2007), 117--126.

\bibitem{Acharya2}	
B.D. Acharya, \emph{Domination in hypergraphs II. New directions}, Proc. Int. Conf.-ICDM 2008, Mysore, India, pp. 1--16.
	
\bibitem{CGCA}	
C. Alfaro, G. Araujo-Pardo, C. Rubio-Montiel and A. V{\' a}zquez-{\' A}vila, \emph{On transversal and 2-packing number in uniform linear systems}, Submitted.
	
\bibitem{CCA}	
C. Alfaro, C. Rubio-Montiel and A. V{\' a}zquez-{\' A}vila, \emph{Covering and 2-packing number in graphs}, https://arxiv.org/abs/1707.02254.

\bibitem{CA}	
C. Alfaro and A. V{\' a}zquez-{\' A}vila, \emph{On a problem of Henning and Yeo about the transversal numbers of uniform linear systems}, https://arxiv.org/abs/1710.02501
	
\bibitem{MR3727901}	
G. Araujo-Pardo, A. Montejano, L. Montejano and A. V{\' a}zquez-{\' A}vila, \emph{On transversal and $2$-packing numbers in straight line systems on $\mathbb{R}^{2}$}, Util. Math. {\bf 105} (2017), 317--336.
	
\bibitem{Arumugam}	
S. Arumugam, B. Jose, Cs. Bujtás and Zs. Tuza, \emph{Equality of domination and transversal numbers in hypergraphs}, Discrete Appl. Math. {\bf 161} (2013), 1859--1867.	

\bibitem{B86}	
L. M. Batten, \emph{Combinatorics of Finite Geometries}, Cambridge Univ Press, Cambridge, 1986.

\bibitem{B95}	
F. Buekenhout, \emph{Handbook of Incidence Geometry: Buildings and Foundations}, Elsevier, 1995.

\bibitem{Buj}	
Cs. Bujtás, M.A. Henning and Zs. Tuza, \emph{Transversals of domination in uniform hypergraphs}, Discrete Appl. Math. {\bf 161} (2013), 1859--1867.


\bibitem{Dong}	
Y. Dong, E. Shan, S. Li and L. Kang, \emph{Domination in intersecting hypergraphs}, Discrete Appl. Math. (2018), https://doi.org/10.1016/j.dam.2018.05.039.


\bibitem{Jose}	
B.K. Jose and Zs. Tuza, \emph{Hypergraph domination and strong independence}, Appl. Anal. Discrete Math. {\bf 3} (2009), 237--358.


\bibitem{Kang}	
L. Kang, S. Li, Y. Dong, E. Shan, \emph{Matching and domination numbers in r-uniform hypergraphs}, J. Comb. Optim. {\bf 34} (2017), 656--659.

\bibitem{Shan}	
E. Shan, Y. Dong, L. Kang and S. Li, \emph{Extremal hypergraphs for matching number and domination number}, Discrete Appl. Math. {\bf 236} (2018), 415--421.
	
\bibitem{Avila}	
A. V{\' a}zquez-{\' A}vila, \emph{Domination and 2-packing number in graphs}, https://arxiv.org/abs/1707.01547
	
\bibitem{Avila2}	
A. V{\' a}zquez-{\' A}vila, \emph{On domination and 2-packing numbers in intersecting linear systems}, Submitted.




	
	
	
	
	
	
	
	
	
	
	
	
	
	
	
	
	
	
	
	
	
	
	
	
	
	
	
	
	
	
	
	
		
\end{thebibliography}
\end{document}